\documentclass{scrartcl}

\usepackage[utf8]{inputenc}
\usepackage[english]{babel}

\usepackage{amsmath}
\usepackage{amssymb}
\usepackage{amsthm}
\usepackage{booktabs}
\usepackage[noadjust]{cite}
\usepackage{enumerate}
\usepackage{float}
\usepackage{ifthen}
\usepackage{url}
\usepackage{adjustbox}
\usepackage{hyperref}
\usepackage{microtype}

\newtheorem{fact}{Fact}[section]
\newtheorem{theorem}{Theorem}[section]
\newtheorem{lemma}[theorem]{Lemma}
\newtheorem{corollary}[theorem]{Corollary}
\theoremstyle{definition}
\newtheorem{definition}[theorem]{Definition}

\providecommand{\FF}{\mathbb{F}}
\providecommand{\RR}{\mathbb{R}}
\providecommand{\cC}{\mathcal{C}}
\providecommand{\cP}{\mathcal{P}}
\providecommand{\scrA}{\mathcal{A}}
\providecommand{\scrC}{\mathcal{C}}
\providecommand{\scrE}{\mathcal{E}}
\providecommand{\scrR}{\mathcal{R}}
\providecommand{\diag}{\text{diag}}

\newcommand{\gauss}[2]{\ifthenelse{\equal{#2}{1}}{[#1]}{\genfrac{[}{]}{0pt}{}{#1}{#2}}}
\newcommand{\gaussstar}[2]{\genfrac{[}{]}{0pt}{}{#1}{#2}}

\newcommand\numberthis{\addtocounter{equation}{1}\tag{\theequation}}

\title{New and Updated Semidefinite Programming Bounds for Subspace Codes}

\author{
Daniel Heinlein\thanks{Department of Communications and Networking, Aalto University, Finland, \href{mailto:daniel.heinlein@aalto.fi}{daniel.heinlein@aalto.fi}. This research was supported by the Academy of Finland, Grant \#289002.},
Ferdinand Ihringer\thanks{Department of Mathematics:  Analysis, Logic and Discrete Mathematics, Ghent University, Belgium, \href{mailto:ferdinand.ihringer@ugent.be}{ferdinand.ihringer@ugent.be} . The author is supported by a postdoctoral fellowship of the Research Foundation --- Flanders (FWO).}
}

\begin{document}
\maketitle

\begin{abstract}
  We show that $A_2(7, 4) \leq 388$ and, more generally, $A_q(7, 4) \leq (q^2-q+1) \gauss{7}{1} + q^4 - 2q^3 + 3q^2 - 4q + 4$
  by semidefinite programming for $q \leq 101$.
  Furthermore, we extend results by Bachoc et al. on SDP bounds for $A_2(n, d)$, where $d$ is odd and $n$ is small, to
  $A_q(n, d)$ for small $q$ and small $n$.


\end{abstract}

\section{Introduction}

By $\cP(V)$ we denote the set of all subspaces in a finite dimensional vector space $V$ over a finite field of order $q$.
The set $\cP(V)$ forms a metric space with respect to the \emph{subspace metric} $d_s(U,W) = \dim(U + W) - \dim(U \cap W)$.
The space $(\cP(V),d_s)$ plays an important role in random linear network coding and was introduced by K\"{o}tter and Kschischang in~\cite{MR2451015} to describe error-detecting and -correcting transmission of informations in the subspace channel model.
A subset $\cC$ of $\cP(V)$ is called \emph{subspace code} and its elements are called \emph{codewords}.
The subspace distance of $\cC$ is given by $d_s(\cC)=\min\{d_s(U,W) : U, W \in V \text{ and } U \ne W\}$.
We refer the reader to Subsection~\ref{subsec:notation} for a more detailed introduction to the used terminology.

The vector $(x_0(\cC), \ldots, x_n(\cC))$ with $x_k(\cC)$ as the number of $k$-subspaces in $\cC$ is called the \emph{dimension distribution} of $\cC$ and the set $K(\cC) = \{ \dim(U) : U \in \cC \}$ contains the dimensions of all codewords of $\cC$.
We drop the reference to $\cC$ if it is clear by the context. Then $(n,N,d;K)_q$ abbreviates the parameters of $\cC$; $\cC \subseteq \cP(\FF_q^n)$, $N=|\cC|$, $d \le d_s(\cC)$, and $K(\cC) \subseteq K$.
If $K(\cC)=\{k\}$, say, then $\cC$ is called \emph{constant-dimension code} (CDC) and is abbreviated as $(n,N,d;k)_q$.
In the other extremal case, i.e., $K=\{0,\ldots,n\}$, the parameters of an (unrestricted) subspace code are abbreviated as $(n,N,d)_q$.

The maximum cardinality $N$ of an $(n,N,d;K)_q$ subspace code is denoted as $A_q(n,d;K)$ and the simpler notation $A_q(n,d;k)$ in the constant-dimension case and $A_q(n,d)$ in the unrestricted case applies, too.
The determination of $A_q(n,d;K)$, or at least suitable bounds, and a classification of all non-isomorphic maximum cardinality codes is known as the \emph{main problem of subspace coding}, since it is the $q$-analog of the \emph{main problem of classical coding theory}, cf.~\cite[Page~23]{MR0465509}. 

The smallest undetermined and arguably most interesting constant-dimension code is a maximum cardinality set of planes in $\FF_2^7$ mutually intersecting in at most a point. Here the best known result is as follows:
\begin{fact}[{\cite[Theorem~2]{Fano333}}]
  We have $333 \le A_2(7,4;3) \le 381$.
\end{fact}
The lower bound was derived by finding a $(7,333,4;3)_2$ CDC after modifying \emph{interesting} codes arising in an exhaustive search in the $\operatorname{GL}(\FF_2^7)$ for subgroups with the property being subgroup of automorphism groups of \emph{large} $(7,N,4;3)_2$ CDCs.
The currently best upper bound is a simple counting argument:
There are $\gauss{7}{2}_2=2667$ lines in $\FF_2^7$, each plane contains $\gauss{3}{2}_2=7$ of them and no line is incident with two codewords, hence $2667/7=381$ upper bounds the size of any $(7,N,4;3)_2$ CDC.
Any putative $(7,381,4;3)_2$ CDC is the binary analog of a Fano plane and a lot of previous work tackle its existence question~\cite{ai2016expurgation,MR3799426,etzion2015structure,kiermaier2016order,liu2014new,MR1305448,MR1419675,MR1725002,MR2589964,MR2796712,MR2801584,MR3198748,MR3398870,MR3403762,MR3444245,MR3468803,MR3542513,MR908977}.

By omitting the constraint on the dimension of codewords, one arrives at $(7,M,4)_2$ subspace codes.
Of course, a $(7,N,4;3)_2$ CDC $\cC$ can be extended to $(7,N+1,4)_2$ subspace code $\cC \cup \{\FF_q^7\}$, providing the best known lower bound $334 \le A_2(7,4)$. Due to Honold et al. we know the following:
\begin{fact}[{\cite[Theorem~4.1]{Honold2016}}]
  We have $A_2(7,4) \leq 407$.
\end{fact}
We improve this to:
\begin{theorem}\label{thm:A_q74_bnd_q_eq_2}
  We have $A_2(7, 4) \leq 388$.
\end{theorem}
If equality holds, then the corresponding code consists up to orthogonality of $41$ lines and $347$ solids (see Lemma~\ref{lem:dist_subspace_code}).
The correspondence to constant-dimension codes shows in particular that a putative binary Fano plane would imply a $(7,382,4)_2$ subspace code and hence reducing the upper bound to less than $382$ would immediately imply the nonexistence of the binary Fano plane -- a seemingly very difficult problem.

In the general case, the best bounds are $q^8+q^5+q^4+q^2-q \le A_q(7,4;3) \le \gauss{7}{2}/\gauss{3}{2} = (q^2-q+1)\gauss{7}{1}$; the lower bound is provided by~\cite[Theorem~4]{MR3444245} and the upper bound arises again by counting lines.
In the unrestricted case, the augmentation of a CDC by $\FF_q^7$ provides again the best known lower bound of $q^8+q^5+q^4+q^2-q+1 \le A_q(7,4)$.
For the upper bound in the unrestricted case, the best previously known method is to relax the minimum distance condition from $4$ to $3$ and then to apply the integer linear programming argument from~\cite[Theorem~10]{MR2810308}.

Define the function $F(q)$ by
\begin{align*}
F(q) = 
\begin{cases}
(q^2-q+1) \gauss{7}{1} + q^4-2q^3+3q^2-4q+3 & \text{ for } q = 2,3,\\
(q^2-q+1) \gauss{7}{1} + q^4-2q^3+3q^2-4q+4 & \text{ for } q \geq 4.
\end{cases}
\end{align*}

\begin{theorem}\label{thm:A_q74_bnd}
  Let $2 \le q \leq 101$ be a prime power. We have $A_q(7, 4) \leq F(q)$.
\end{theorem}

This gives $388, 7696, 71157, 410585$ for $q=2,3,4,5$, while the previous best known bounds
were $407, 15802, 144060, 826594$.
The bound $q \leq 101$ is chosen rather arbitrarily and we conjecture that it is unnecessary.
For general $q$, we could only show the following.

\begin{theorem}\label{thm:A_q74_bnd_gen}
    Let $2 \le q$ be a prime power. We have $A_q(7, 4) \leq (q^2-q+1) \gauss{7}{1} + 2(q^5+q^3+1)$.
\end{theorem}

Previously, Bachoc et al. applied semidefinite programming in~\cite{Bachoc2013} to binary subspace codes with odd minimum distance and $n \le 16$. We extend their results in several ways: (1) Since Bachoc et al. computed their bounds,
several new upper bounds for small CDC codes were discovered, cf. \url{http://subspacecodes.uni-bayreuth.de/} associated with~\cite{HKKW2016Tables}. Using these new bounds, we provide an update on their
bounds (with a slightly differently chosen range of parameters). (2) We provide
bounds for $d$ even. (3) We compute bounds for $q > 2$. Our range for all these computations is mostly arbitrary,
but chosen in a way that the computations terminate in less than a week on standard hardware at the time of writing.

The paper is organized as follows. In Section~\ref{sec:prelim} we introduce basic definitions and the used theoretical framework
of semidefinite programming in coherent configurations, so that we can describe the coherent configuration and semidefinite program
which is associated with the symmetry group of the metric space $(\cP(V),d_s)$ in Section~\ref{sec:cc}.
This culminates in Section~\ref{sec:main_thm}, in which we investigate $A_q(7, 4)$ and show our main results,
and Section~\ref{sec:comps}, in which we update the SDP bounds given by Bachoc et al. 
To conclude this current overview on semidefinite programming for subspace codes,
we mention that Schrijver's SDP bound does not appear to improve any known bounds and we suggest some future work.

\section{Preliminaries}\label{sec:prelim}

\subsection{Subspace Codes} \label{subsec:notation}

Let $2 \le q$ be a prime power, $\FF_q$ the field with $q$ elements, and $V \cong \FF_q^n$ the $n$-dimensional vector space over $\FF_q$.
By $\cP(V)$ we denote the set of all subspaces in $V$.
For two subspaces $U,W \in \cP(V)$ we write $U \le W$ iff $U$ is subspace of $W$.
Recall that $\cP(V)$ forms a metric space with respect to the \emph{subspace metric}~\cite[Section~3.1]{MR2451015}
\begin{align*}d_s(U,W) = \dim(U + W) - \dim(U \cap W).\end{align*}

For $k \in \{0,1,\ldots,v\}$, $\gauss{V}{k}$ denotes the set of $k$-dimensional subspaces in $V$.
Its cardinality is given by the \emph{$q$-binomial coefficient}
\begin{align*}|\gauss{V}{k}| = \gauss{n}{k}_q = \prod_{i=1}^{k} \frac{q^{n-k+i}-1}{q^{i}-1}.\end{align*}
As an abbreviation we use the \emph{$q$-number} $\gauss{n}{1}_q=\gaussstar{n}{1}_q$ and drop the index $q$ in $\gauss{n}{1}_q$ and $\gauss{n}{k}_q$ if there is no confusion with $\gauss{V}{k}$ and $q$ is clear by the context.
Using the \emph{$q$-factorial} $\gauss{n}{1}! = \prod_{i=1}^n \gauss{i}{1}$, the $q$-binomial coefficient can then be expressed as $\gauss{n}{k} = \frac{\gauss{n}{1}!}{\gauss{k}{1}!\gauss{n-k}{1}!}$.
A $k$-dimensional subspace of $V$ is called simply $k$-subspace and we refer to $1$-subspaces as points, $2$-subspaces as lines, $3$-subspaces as planes, $4$-subspaces as solids, and $(n-1)$-subspaces as hyperplanes.

Let $\cC$ be a subspace code. Recall that for $2 \le |\cC|$ the subspace distance of $\cC$ is given by $d_s(\cC)=\min\{d_s(U,W) : U, W \in V \text{ and } U \ne W\}$ and notice that we formally set $d_s(\cC)=\infty$ if $|\cC| \le 1$.

By $x_i(\cC)$ we denote the number of $i$-subspaces in $\cC$ and drop the reference to $\cC$ if it is clear from the context.

The automorphism group of $(\cP(V),d_s)$ for $3 \le n$ was shown to be generated by $\operatorname{P\Gamma{}L}(V)$ and a polarity $\pi : \cP(V) \rightarrow \cP(V), U \mapsto U^\perp$ (see e.g.~\cite[Theorem~2.1]{Honold2016}).
We call $U^\perp$ the \emph{orthogonal space} of $U$ and apply $\pi$ also to subspace codes $\cC$ to obtain their \emph{orthogonal codes} $\cC^\perp$.
If $\cC$ is an $(n,N,d;K)_q$ subspace code with dimension distribution $(x_0(\cC), \ldots, x_n(\cC))$, then $\cC^\perp$ is an $(n,N,d;\{n-i : i \in K\})_q$ subspace code with dimension distribution $(x_n(\cC), \ldots, x_0(\cC))$, in particular $A_q(n,d;k) = A_q(n,d;n-k)$.

\subsection{Coherent Configurations}

We follow the notation and point of view by Hobart and Williford for applying a semidefinite programming bound 
which is set in the context of coherent configurations and we refer to their work for a general introduction
to that topic~\cite{Higman1975,Higman1976,Hobart2009,Hobart2014}.

\begin{definition}
 Let $X$ be a finite set. A coherent configuration is a pair $(X, \scrR)$,
 where $\scrR = \{ R_0, \ldots, R_l \}$ is a set of binary relations
 on $X$ with the following properties:
  \begin{enumerate}[(a)]
   \item $\scrR$ is a partition of $X \times X$.
   \item If $R_i \cap \diag(X \times X) \neq \emptyset$, then $R_i \subseteq \diag(X \times X)$.
   \item If $R_i \in \scrR$, then $R_i^T \in \scrR$.
   \item For $R_i, R_j, R_k \in \scrR$ and $x,y \in X$ with $(x, y) \in R_k$, the number of $z$ such
   that $(x, z) \in R_i$ and $(z, y) \in R_j$ is a constant $p_{ij}^k$, independent of the choice of $x$ and $y$.
  \end{enumerate}
\end{definition}
These $p_{ij}^k$ are commonly called \emph{intersection numbers}.
Condition (b) gives a partition of the identity relation into sets $X_a$ called \textit{fibers}.
In the group case, i.e., a group $G$ operating on the finite set $X$, the induced component-wise action of $G$ on $X \times X$ yields a coherent configuration in which the relations are given by the orbits of $G$ on $X \times X$, cf.~\cite[Pages~212 and~217]{MR898557}.
Each relation is contained in some $X_{a} \times X_{a'}$. If we restrict $X$ to some $X_a$, then we obtain a (homogeneous) \textit{association scheme}.
For each $R_i$ we can define an $|X| \times |X|$ matrix $A_i$ indexed by $X$ with
\begin{align*}
  (A_i)_{xy} = 
  \begin{cases}
    1 & \text{ if } (x, y) \in R_i,\\
    0 & \text{ otherwise.}
  \end{cases}
\end{align*}
The matrices $\{ A_0, \ldots, A_l \}$ generate an algebra $\scrA$ with several useful properties.
For the representation theory of $\scrA$ we follow the notation of~\cite{Hobart2014}.
Let $\{ \Delta_1, \ldots, \Delta_m \}$ the set of absolutely irreducible representations 
of $\scrA$, chosen such that $\Delta_s(A^*) = (\Delta_s(A))^*$. Denote the multiplicity 
of $\Delta_s$ by $f_s$. Let $\gamma$ denote the number of fibers of the coherent configuration
and $E_{ij}$ the $(\gamma \times \gamma)$-matrix with a $1$ at position $(i, j)$ and $0$ otherwise.
Since $\scrA$ is semisimple, it decomposes into a direct sum of algebras $\scrE_s$.
There exists a basis $\scrE^s_{ij}$ for each algebra $\scrE_s$ satisfying the following equations:
\begin{align}
  \scrE^s_{ij} \scrE^t_{kl} = \delta_{st} \delta_{jk} \scrE^s_{il},&&
  (\scrE^s_{ji})^* = \scrE^s_{ij}, \quad \text{ and}&&
  \Delta_s(\scrE^t_{ij}) = \delta_{st} E_{ij}.
  \label{eq_scrE}
\end{align}
Let $m_i = |R_i|$. Then
\begin{align}
  &A_k = \sum_{i,j,s} (\Delta_s(A_k))_{ij} \scrE_{ij}^s && \text{and} && \scrE_{ij}^s = f_s \sum_k \frac{1}{m_k} \overline{(\Delta_s(A_k))_{ij}} A_k.\label{eq:dual_P_Q}
\end{align}

The next lemma shows bounds on subsets of $X$ in terms of the positive semidefiniteness of involved matrices.
Bounds arising by this method are commonly called semidefinite programming bound as it is a generalization of Delsarte's linear programming bound~\cite{Delsarte1973}.

\begin{theorem}[{\cite[Theorem~2.2 and~2.3]{Hobart2009}}] \label{thm:sdp}
  Let $(X, \scrR)$ be a coherent configuration, $Y \subseteq X$, and $b_i = |(Y \times Y) \cap R_i|$.
  Define $D(Y) = \sum_{i=1}^l \frac{b_i}{m_i} A_i$. Then the matrices $D(Y)$ and $\Delta_s(D(Y))$
  are positive semidefinite for any irreducible
  representation $\Delta_s$ of the coherent configuration satisfying $\Delta_s(A^*) = (\Delta_s(A))^*$.
\end{theorem}

If all fibers of a coherent configuration correspond to a commutative association scheme, we can use the the intersection numbers, i.e., the algebra generated by the \emph{intersection matrices} $L_i=(p_{ij}^k)_{kj}$, to first calculate all $\scrE^s_{ij}$ via the eigenvalues of the association scheme restricted to the fibers (see~\cite[Chapter~2, Proposition~2.2.2]{Brouwer1989}) and then apply the identities~\eqref{eq_scrE} to determine the remaining parameters. In Section~\ref{sec:technique} we provide details for this calculation.

\subsection{Semidefinite programming}

We abbreviate the term positive semidefinite as psd and for symmetric matrices $A$ and $B$ we write $A \succcurlyeq B$ iff $A-B$ is psd.
A \emph{semidefinite program} (SDP) is an optimization problem of the form
\begin{align*}
\min c^T &x \numberthis \label{SDPPRIMAL} \\
\text{subject to } \sum_{i=1}^{m} F_i &x_i \succcurlyeq F_0 \\
&x \in \RR^m
\end{align*}
with $c \in \RR^m$ and symmetric $F_i \in \RR^{n \times n}$ for $i \in \{0,\ldots,m\}$.
The dual problem associated with~\eqref{SDPPRIMAL} (which is then called primal) is
\begin{align*}
\max \operatorname{tr}(F_0 &Z) \\
\text{subject to } \operatorname{tr}(F_i &Z) = c_i\text{ for all } i \in\{1,\ldots,m\} \\
&Z \succcurlyeq 0
\end{align*}
and, if the primal and dual contain feasible points $x$ and $Z$, the optimal value of the dual lower bounds the optimal value of the primal.
We have equality if the primal or the dual contains strictly feasible points, cf.~\cite[Page~64 and Theorem~3.1]{MR1379041}.
Although it can be solved in polynomial time with the ellipsoid method, interior-points methods are often faster in practice cf.~\cite[Page~52]{MR1379041} and~\cite{MR2894706}.

Using the Schur complement, many quadratic inequalities can be modeled as constraints in an SDP:
Let $M=\left(\begin{smallmatrix}A&B\\B^T&C\end{smallmatrix}\right)$ be symmetric and $A$ be positive definite, then $M$ is psd iff $C-B^T A^{-1} B$ is psd.
In particular, using $I$ as an identity matrix of appropriate size, $\left(\begin{smallmatrix}I&Ax-b\\(Ax-b)^T&c^Tx-d\end{smallmatrix}\right)$ is positive semidefinite iff $(Ax-b)^T(Ax-b) \le c^Tx-d$.

If multiple matrices shall be psd simultaneously, they are commonly arranged as blocks on the main diagonal of the $F_i$ and linear inequalities are commonly embedded as diagonal matrices, hence any linear program can be written as an SDP.

\section{The Coherent Configuration of \texorpdfstring{$\operatorname{P\Gamma{}L}(V)$}{PGammaL(V)} operating on \texorpdfstring{$\cP(V)$}{P(V)}}\label{sec:cc}

\subsection{Triples in Vector Spaces}\label{sec:triples}

In this section we provide a general formula for counting triples in vector spaces.

\begin{lemma}
  Let $A$ be an $a$-space and $B$ a $b$-space with $c = \dim(A \cap B)$ in $\FF_q^{a+b-c}$.
  Then the number of $d$-spaces $D$ having trivial intersection with $A$ and $B$ is
  \begin{align*}
    \psi(a, b, c, d) := \prod_{j=0}^{d-1} \frac{q^{j+c}(q^{a-c-j}-1)(q^{b-c-j}-1)}{q^{d-j} - 1}.
  \end{align*}
\end{lemma}
\begin{proof}
  We double count $((P_0, \ldots, P_{d-1}), D)$, where $(P_0, \ldots, P_{d-1})$ is an ordered basis of $D$.
  For $P_0, \ldots, P_{j-1}$ given, we have 
  \begin{align*}
   \gauss{a+b-c}{1} - \gauss{a+j}{1} - \gauss{b+j}{1} + \gauss{c+2j}{1} = \frac{q^{2j+c}(q^{a-c-j}-1)(q^{b-c-j}-1)}{q-1}
  \end{align*}
  choices for $P_j$. Hence, we have $\prod_{j=0}^{d-1} \frac{q^{2j+c}(q^{a-c-j}-1)(q^{b-c-j}-1)}{q-1}$ choices for $(P_0, \ldots, P_{d-1})$.
  Similarly, the number of choices for $(P_0, \ldots, P_{d-1})$ with given $D$ is $\prod_{j=0}^{d-1} ([d]-[j]) = \prod_{j=0}^{d-1} \frac{q^j(q^{d-j} - 1)}{q-1}$, showing the assertion.
\end{proof}

\begin{lemma}
  Let $A$ be an $a$-space and $B$ a $b$-space with $c = \dim(A \cap B)$ in $\FF_q^{a+b-c}$.
  Then the number of $d$-spaces $D$ meeting $A$ in an $\alpha$-space, $B$ in an $\beta$-space
  and $A \cap B$ in a $\gamma$-space is $\varphi(a, b, c, d, \alpha, \beta, \gamma) := $
  \begin{align*}
    \gauss{c}{\gamma} q^{(\alpha+\beta-2\gamma)(c - \gamma)} \gauss{a - c}{\alpha - \gamma} \gauss{b - c}{\beta - \gamma} \psi(a - \alpha, b - \beta, c - \gamma, d - \alpha - \beta + \gamma).
  \end{align*}
\end{lemma}
\begin{proof}
  Clearly, there are $\gauss{c}{\gamma}$ choices for $A \cap B \cap D$. It is well-known that the remaining choices for $A \cap D$ and $B \cap D$
  are 
  \begin{align*}
    q^{(\alpha+\beta-2\gamma)(c - \gamma)} \gauss{a - c}{\alpha - \gamma} \gauss{b - c}{\beta - \gamma}.
  \end{align*}
  In the quotient of $\langle A \cap D, B \cap D\rangle$ we see that we have $\psi(a - \alpha, b - \beta, c - \gamma, d - \alpha - \beta + \gamma)$ choices left to complete $D$.
\end{proof}

Now we obtain the following:
\begin{lemma}\label{lem:int_numbers}
  Let $A$ be an $a$-space and $B$ a $b$-space with $c = \dim(A \cap B)$ in $\FF_q^{n}$.
  Then the number of $d$-spaces $D$ meeting $A$ in an $\alpha$-space, $B$ in an $\beta$-space
  and $A \cap B$ in a $\gamma$-space is
  \begin{align*}
    \chi(a, b, c, d, n, \alpha, \beta, \gamma) := \sum_{x= \alpha+\beta-\gamma}^{\min\{ d, a+b-c \}} q^{(d-x)(a+b-c-x)} \gauss{n-a-b+c}{d-x} \varphi(a, b, c, x, \alpha, \beta, \gamma).
  \end{align*}
\end{lemma}

Hence, we conclude that we can count triples as follows.

\begin{lemma}\label{lem:int_vals}
  Let $A$ be an $a$-space and $B$ a $b$-space which meet in codimension $k$ in $\FF_q^{n}$.
  Then the number of $d$-spaces $D$ meeting $A$ in codimension $i$ and $B$ in codimension $j$ is
  \begin{align*}
    \sum_{\ell = 0}^{ \min\{a, b\}-k } \chi( a, b, \min\{ a, b \} - k, d, n, \min\{ a, d \} - i, \min\{ b, d \} - j, \min\{a, b\} - k - \ell).
  \end{align*}
\end{lemma}

Since each relation is contained in some $X_a \times X_b$ we index the relations, basis matrices, intersection numbers, etc. accordingly: $R_{ab\ell}$, $A_{ab\ell}$, $p_{(a,d,i),(d,b,j)}^{(a,b,k)}$, $m_{ab\ell}$, and $b_{ab\ell}$ such that $a,b,d$ are indices of fibers and $\ell,k,i,j$ are counters.
In particular, all other intersection numbers are zero.
The first equation of the identities (\ref{eq:dual_P_Q}) is hence $A_{ab\ell} = \sum_{s} (\Delta_{s}(A_{ab\ell}))_{ab} \scrE_{ab}^{s}$.

The intersection numbers $p_{(a,d,i),(d,b,j)}^{(a,b,k)}$ are given by the expression in the last lemma and all other intersection numbers vanish.

\subsection{Irreducible Representations}

\begin{table}
\begin{center}
\begin{adjustbox}{max width=\textwidth}
\begin{tabular}{llllll}
$A_{abc}$ & $m_{abc}/|X_a|$ & $\Delta_0(A_{abc})$ & $\Delta_1(A_{abc})$ & $\Delta_2(A_{abc})$ & $\Delta_3(A_{abc})$\\\midrule
$A_{110}$ & $1$         & $E_{11}$              & $E_{11}$ &        &  \\
$A_{111}$ & $q\gauss{6}{1}$     & $q\gauss{6}{1} E_{11}$          & $-E_{11}$ &       &  \\
$A_{120}$ & $\gauss{6}{1}$      & $\gauss{2}{1} \sqrt{ \psi\gauss{3}{1} } E_{12}$     & $\sqrt{q \gauss{5}{1} } E_{12}$ &       &  \\
$A_{121}$ & $q^2\psi \gauss{3}{1} \gauss{5}{1}$ & $q^2 \gauss{5}{1} \sqrt{ \psi\gauss{3}{1} } E_{12}$     & $-\sqrt{q \gauss{5}{1} } E_{12}$ &        &  \\
$A_{130}$ & $\gauss{6}{2}$      & $\gauss{3}{1} \sqrt{ \psi \gauss{5}{1} } E_{13}$  & $q \sqrt{ \varphi \gauss{5}{1}} E_{13}$ &       &  \\
$A_{131}$ & $q^3(q^3+1)\gauss{5}{2}$    & $q^3 \gauss{4}{1} \sqrt{ \psi \gauss{5}{1} } E_{13}$    & $-q \sqrt{ \varphi \gauss{5}{1}}E_{13}$ &       &  \\
$A_{140}$ & $\gauss{6}{3}$      & $\gauss{4}{1} \sqrt{ \psi \gauss{5}{1} } E_{14}$    & $q \sqrt{q \varphi \gauss{5}{1}} E_{14}$ &        &  \\
$A_{141}$ & $q^4 \psi \gauss{3}{1} \gauss{5}{1}$& $q^4 \gauss{3}{1} \sqrt{ \psi \gauss{5}{1} } E_{14}$    & $-q \sqrt{q \varphi \gauss{5}{1}} E_{14}$ &       &  \\
$A_{150}$ & $\gauss{6}{4}$      & $\gauss{5}{1} \sqrt{ \psi \gauss{3}{1} } E_{15}$    & $q^2 \sqrt{ \gauss{5}{1} } E_{15}$ &        &  \\
$A_{151}$ & $q^5\gauss{6}{1}$     & $q^5 \gauss{2}{1} \sqrt{ \psi \gauss{3}{1} } E_{15}$    & $-q^2 \sqrt{ \gauss{5}{1} } E_{15}$ &       &  \\
$A_{160}$ & $\gauss{6}{5}$      & $\gauss{6}{1} E_{16}$           & $q^{5/2} E_{16}$ &        &  \\
$A_{161}$ & $q^6$       & $q^6 E_{16}$              & $-q^{5/2} E_{16}$ &       &  \\
$A_{220}$ & $1$         & $E_{22}$              & $E_{22}$ & $E_{22}$       &  \\   
$A_{221}$ & $q\gauss{2}{1} \gauss{5}{1}$   & $q\gauss{2}{1} \gauss{5}{1} E_{22}$          & $(q^2 \gauss{4}{1} - 1)E_{22}$  & $-\gauss{2}{1}E_{22}$ &  \\
$A_{222}$ & $q^4\varphi \gauss{5}{1}$   & $q^4\varphi \gauss{5}{1} E_{22}$        & $-q^2 \gauss{4}{1} E_{22}$    & $qE_{22}$ & \\
$A_{230}$ & $\gauss{5}{1}$      & $\sqrt{ \gauss{3}{1} \gauss{5}{1} } E_{23}$       & $\gauss{2}{1} \sqrt{q\varphi} E_{23}$ & $ q \sqrt{ \gauss{3}{1} } E_{23}$ & \\
$A_{231}$ & $q^2 \gauss{4}{1}\gauss{5}{1}$  & $q^2  \gauss{4}{1} \sqrt{ \gauss{3}{1} \gauss{5}{1} } E_{23}$ & $(q^3\gauss{3}{1} - \gauss{2}{1}) \sqrt{q\varphi} E_{23}$ & $-q\gauss{2}{1} \sqrt{ \gauss{3}{1} } E_{23}$ & \\
$A_{232}$ & $q^6\varphi\gauss{5}{1}$    & $q^6 \varphi \sqrt{ \gauss{3}{1} \gauss{5}{1} } E_{23}$   & $-q^3\gauss{3}{1} \sqrt{q\varphi} E_{23}$ & $ q^2\sqrt{ \gauss{3}{1} } E_{23}$ & \\
$A_{240}$ & $\varphi\gauss{5}{1}$   & $\varphi \sqrt{ \gauss{3}{1} \gauss{5}{1} } E_{24}$     & $q \gauss{3}{1} \sqrt{\varphi} E_{24}$ & $q^2 \sqrt{\gauss{3}{1}}E_{24}$ & \\
$A_{241}$ & $q^3 \gauss{4}{1}\gauss{5}{1}$  & $q^3 \gauss{4}{1} \sqrt{ \gauss{3}{1} \gauss{5}{1} } E_{24}$  & $q (q^4\gauss{2}{1}-\gauss{3}{1}) \sqrt{\varphi} E_{24}$ & $-q^2\gauss{2}{1} \sqrt{\gauss{3}{1}} E_{24}$ & \\
$A_{242}$ & $q^8 \gauss{5}{1}$      & $q^8 \sqrt{ \gauss{3}{1} \gauss{5}{1} } E_{24}$     & $-q^5 \gauss{2}{1} \sqrt{\varphi} E_{24}$ & $q^3 \sqrt{\gauss{3}{1}} E_{24}$ & \\
$A_{250}$ & $\varphi \gauss{5}{1}$    & $\varphi \gauss{5}{1} E_{25}$         & $q^{3/2} \gauss{4}{1} E_{25}$         & $q^3 E_{25}$      &  \\   
$A_{251}$ & $q^4 \gauss{2}{1} \gauss{5}{1}$ & $q^4 \gauss{2}{1} \gauss{5}{1} E_{25}$      & $q^{3/2}(q^5 - \gauss{4}{1}) E_{25}$        & $-\gauss{2}{1} q^3E_{25}$ &  \\
$A_{252}$ & $q^{10}$        & $q^{10} E_{25}$           & $-q^{13/2} E_{25}$            & $q^4 E_{25}$ & \\
$A_{330}$ & $1$         & $E_{33}$              & $E_{33}$ & $E_{33}$ & $E_{33}$\\   
$A_{331}$ & $q \gauss{3}{1} \gauss{4}{1}$ & $q \gauss{3}{1} \gauss{4}{1} E_{33}$        & $(q^2\gauss{2}{1}\gauss{3}{1} - 1)E_{33}$           & $(q^2-1) \gauss{3}{1} E_{33}$   & $-\gauss{3}{1} E_{33}$\\
$A_{332}$ & $q^4 \varphi \gauss{3}{1}^2$  & $q^4 \varphi \gauss{3}{1}^2 E_{33}$         & $q^2 \gauss{3}{1} (q^4-q-1)E_{33}$          & $-q\gauss{3}{1} (q^2+q-1)E_{33}$  & $q \gauss{3}{1} E_{33}$\\
$A_{333}$ & $q^9\gauss{4}{1}$     & $q^9\gauss{4}{1} E_{33}$          & $-q^6 \gauss{3}{1} E_{33}$              & $q^4 \gauss{2}{1}E_{33}$      & $-q^3 E_{33}$\\
$A_{340}$ & $\gauss{4}{1}$      & $\gauss{4}{1} E_{34}$           & $\gauss{3}{1} \sqrt{q} E_{34}$        & $q \gauss{2}{1} E_{34}$ & $\sqrt{q^3 } E_{34}$\\
$A_{341}$ & $q^2\varphi\gauss{3}{1}^2$    & $q^2 \varphi \gauss{3}{1}^2 E_{34}$         & $\gauss{3}{1} (q^3\gauss{2}{1}-1) \sqrt{q} E_{34}$  & $q \gauss{3}{1}(q^2-q-1) E_{34}$ & $-\gauss{3}{1} \sqrt{q^3 } E_{34}$\\
$A_{342}$ & $q^6\gauss{3}{1} \gauss{4}{1}$  & $q^6 \gauss{3}{1} \gauss{4}{1} E_{34}$      & $q^3(q^5-\gauss{2}{1}\gauss{3}{1}) \sqrt{q} E_{34}$ & $-q^2 (q^3-1) \gauss{2}{1} E_{34}$ & $q\gauss{3}{1} \sqrt{q^3 }E_{34}$\\
$A_{343}$ & $q^{12}$        & $q^{12} E_{34}$             & $-q^8 \sqrt{q} E_{34}$      & $q^6 E_{34}$ & $-q^3\sqrt{q^3 }E_{34}$\\
$f_s$ &           & $1$   & $\gauss{7}{1}-1$ & $\gauss{7}{2} - \gauss{7}{1}$ & $\gauss{7}{3} - \gauss{7}{2}$\\
\end{tabular}
\end{adjustbox}
\caption{Here $\varphi=q^2+1$ and $\psi = q^2-q+1$.}
\label{table:irred_reps}
\end{center}
\end{table}

The coherent configuration in this paper arises by the action of $\operatorname{P\Gamma{}L}(V)$ on $\cP(V) \times \cP(V)$.
Hence, we have the $n+1$ fibers labeled with $0,1, \ldots, n$, such that the $k$-th fiber consists of all $k$-spaces of $V$.
A pair of subspaces $(x,y)$ is in the relation $R_{abc}$ iff $x$ has dimension $a$, $y$ has dimension $b$, and $c = \min\{a,b\} - \dim(x \cap y)$ for all $a,b \in \{ 0, \ldots, n+1\}$ and $c \in \{ 0, \ldots, \min\{\min\{a,b\}, n-\min\{a,b\}\} \}$.
The benefit of choosing $c$ as the codimension of the intersection is that $R_{ii0}$ corresponds to the identity on the $i$-th fiber.
The fibers of this coherent configuration are obviously symmetric association schemes and hence by~\cite[Chapter~4]{Higman1975} commutative.
For $V\cong \FF_q^7$, we show in Corollary~\ref{cor:dim0167} that the 0-space and the 7-space cannot be contained in a large subspace code and hence we restrict ourself in this case to proper subspaces.

Since we investigate the bound on $A_q(7, 4)$ analytically, Table~\ref{table:irred_reps} shows the representation explicitly in the style of Hobart and Williford~\cite{Hobart2014}.
To improve the notation, we also introduce the abbreviations $\varphi=q^2+1$ and $\psi = q^2-q+1$.
Notice that
\begin{align}
|X_a| = \gauss{7}{a}, && \Delta_s(A_{xyc})=E_{xa}\Delta_s(A_{abc})E_{by}, \quad \text{ and} && m_{xyc}=m_{abc} \label{eq:XAm}
\end{align}
for $(x,y) \in \{(a,b),(b,a),(7-a,7-b),(7-b,7-a)\}$ by orthogonality and symmetry for all $a$, $b$, $c$, and $s$.

\subsection{Calculating the Irreducible Representation}\label{sec:technique}

Let us outline how to calculate $\Delta_s$.
Since our fibers are commutative, we can use standard techniques for commutative association schemes, see~\cite[Prop.~2.2.2]{Brouwer1989},
to calculate
\begin{align*}
  A_{iik} = \sum_{s} (\Delta_s(A_{iik}))_{ii} \scrE_{ii}^s.
\end{align*}
This yields the values of $\Delta_s(A_{iik})$.
Next, we will use
\begin{align*}
\Delta_s(A_{ijk})^T
= \Delta_s(A_{ijk})^{**T}
= \Delta_s(A_{ijk}^*)^{*T}
= \overline{\Delta_s(A_{jik})}
\end{align*}
in an application of Eq.~\eqref{eq_scrE} and Eq.~\eqref{eq:dual_P_Q}, so that
\begin{align*}
 A_{ijk}A_{jit} &= \left( \sum_s (\Delta_s(A_{ijk}))_{ij} \scrE_{ij}^s \right) \left( \sum_u (\Delta_u(A_{jit}))_{ji} \scrE_{ji}^u \right)\\
  &= \sum_s (\Delta_s(A_{ijk}))_{ij} (\Delta_s(A_{jit}))_{ji} \scrE_{ij}^s \scrE_{ji}^s = \sum_s (\Delta_s(A_{ijk}))_{ij}\overline{(\Delta_s(A_{ijt}))_{ij}} \scrE_{ii}^s.
\end{align*}
From the triple intersection numbers, we have
\begin{align*}
  A_{ijk}A_{jit} &= \sum_{\ell} p^{(i,i,\ell)}_{(i,j,k),(j,i,t)} A_{ii\ell}
  = \sum_{\ell,s} p^{(i,i,\ell)}_{(i,j,k),(j,i,t)} (\Delta_s(A_{ii\ell}))_{ii} \scrE_{ii}^s.
\end{align*}
Hence, we obtain
\begin{align}
\label{eq_Delta_general}
(\Delta_s(A_{ijk}))_{ij}\overline{(\Delta_s(A_{ijt}))_{ij}} = \sum_{\ell} p^{(i,i,\ell)}_{(i,j,k),(j,i,t)} (\Delta_s(A_{ii\ell}))_{ii}
\end{align}
and, in particular for $k=t$,
\begin{align}
\label{eq_Delta_special}
|(\Delta_s(A_{ijk}))_{ij}| = \sqrt{\sum_{\ell} p^{(i,i,\ell)}_{(i,j,k),(j,i,k)} (\Delta_s(A_{ii\ell}))_{ii}}.
\end{align}
Note, that the right hand sides of the Eq.~\eqref{eq_Delta_general} and~\eqref{eq_Delta_special} are known since only information of the association schemes are involved.
Furthermore, if for fixed $s,i,j$ one nonzero $(\Delta_s(A_{ijk}))_{ij}$ for any $k$ is chosen, then the other $(\Delta_s(A_{ijt}))_{ij}$ ($t \ne k$) are uniquely determined as
\begin{align}
\label{eq_Delta_afterchosen}
(\Delta_s(A_{ijt}))_{ij} = \frac{\sum_{\ell} p^{(i,i,\ell)}_{(i,j,k),(j,i,t)} (\Delta_s(A_{ii\ell}))_{ii}}{\overline{(\Delta_s(A_{ijk}))_{ij}}},
\end{align}
here we are using the fact that the $(\Delta_s(A_{ii\ell}))_{ii}$ are algebraic integers, cf.~\cite[Page~45]{Brouwer1989}.

Then, for fixed $s,i,j$, we use Eq.~\eqref{eq_Delta_special} to determine that either $(\Delta_s(A_{ijk}))_{ij}=0$ for all $k$ or we choose one arbitrary $(\Delta_s(A_{ijk}))_{ij} \ne 0$ such that it fulfills Eq.~\eqref{eq_Delta_special}.
Next, in the second case, we apply Eq.~\eqref{eq_Delta_afterchosen} to determine $(\Delta_s(A_{ijt}))_{ij}$ ($t \ne k$).

In particular, we are choosing $(\Delta_s(A_{ijk}))_{ij} \ne 0$ to be a real number and hence we get the simplified formulas
\begin{align}
\label{eq_Delta_summaryreal_special}
(\Delta_s(A_{ijk}))_{ij} &= \pm\sqrt{\sum_{\ell} p^{(i,i,\ell)}_{(i,j,k),(j,i,k)} (\Delta_s(A_{ii\ell}))_{ii}} \quad\text{ and} \\
\label{eq_Delta_summaryreal_afterchosen}
(\Delta_s(A_{ijt}))_{ij} &= \frac{\sum_{\ell} p^{(i,i,\ell)}_{(i,j,k),(j,i,t)} (\Delta_s(A_{ii\ell}))_{ii}}{(\Delta_s(A_{ijk}))_{ij}} \quad \text{ for } t \ne k.
\end{align}

\subsection{Example: \texorpdfstring{$\cP(\mathbb{F}_q^3)$}{P(Fq3)}}
We consider $1$- and $2$-spaces in $\cP(\mathbb{F}_q^3)$.
We abbreviate $I$ and $J$ as the identity respective the all-one matrix of size $n=q^2+q+1$.
As $A_{110}=I$, its only eigenvalue is $1$ and as $A_{111}=J-I$ is the adjacency matrix of a complete graph on $n$ vertices, its eigenvalues are $q^2+q$ and $-1$.
That is $\scrE_{11}^0 = J/n$, $\scrE_{11}^1 = I-J/n$, $A_{110} = \scrE_{11}^0 + \scrE_{11}^1$, and $A_{111} = (q^2+q)\scrE_{11}^0 - \scrE_{11}^1$.

Since $(A_{120})_{P,L}=1$ iff $P \le L$ and $A_{210}^T = A_{120}$, we have that $(A_{120}A_{210})_{P,Q}$ is $1$ if $P \ne Q$ and $q+1$ else, i.e.,
\begin{align*}
A_{120}A_{210} = (q+1) A_{110} + A_{111} = (q+1)^2 \scrE_{11}^0 + q \scrE_{11}^1.
\end{align*}
Eq.~\eqref{eq_Delta_summaryreal_special} implies $(\Delta_0(A_{120}))_{12} = \pm\sqrt{(q+1)^2}$ and $(\Delta_1(A_{120}))_{12} = \pm\sqrt{q}$.
We choose $(\Delta_0(A_{120}))_{12} = q+1$ and $(\Delta_1(A_{120}))_{12} = \sqrt{q}$.

Applying Eq.~\eqref{eq_Delta_summaryreal_afterchosen} and $A_{120} A_{211} = q A_{111} = q(q^2+q) \scrE_{11}^0 -q \scrE_{11}^1$ yield
\begin{align*}
(\Delta_0(A_{121}))_{12} = \frac{q (q^2+q)}{q+1} = q^2 \quad \text{ and } \quad (\Delta_1(A_{121}))_{12} = \frac{-q}{\sqrt{q}} = -\sqrt{q}.
\end{align*}

\subsection{Semidefinite programming}\label{subsec:sdp}

We apply Theorem~\ref{thm:sdp} for $(n,|\cC|,d)_q$ subspace codes $\cC \subseteq \cP(V)$.
Then $b_{ijl} = |(\cC \times \cC) \cap R_{ijl}|$ is the number of pairs $(U,W)$ of codewords in $\cC$ such that $\dim(U)=i$, $\dim(W)=j$, and $\min\{i,j\}-\dim(U \cap W)=l$.
The minimum subspace distance of $d$ implies that $b_{ijl}=0$ for triples $i,j,l$ satisfying $i \ne j$ or $1 \le l$ if $l < \min\{i,j\}+(d-i-j)/2$.
In particular, the number of $i$-subspaces in $\cC$ is given by $x_i=b_{ii0}$ and they fulfill
\begin{align}
b_{ijl} = b_{jil},&&
b_{ii0}^2 = \sum_l b_{iil}, \quad \text{ and}&&
b_{ii0}b_{jj0} = \sum_l b_{ijl}. \label{bij}
\end{align}

Since the last two conditions of Equations~(\ref{bij}) cannot be expressed as constraints in an SDP, we implement only two inequalities:
First, $b_{ii0}^2 \le \sum_l b_{iil}$ corresponds via the Schur complement to $\left(\begin{smallmatrix}1&b_{ii0}\\b_{ii0}&\sum_l b_{iil}\end{smallmatrix}\right) \succcurlyeq 0$.
Second, $b_{ii0} b_{jj0} \ge \sum_l b_{ijl}$ is equivalent to $b_{ii0}^2 b_{jj0}^2 \ge (\sum_l b_{ijl})^2$ and using Equations~(\ref{bij}) this is again equivalent to $\left(\begin{smallmatrix}\sum_l b_{iil}&\sum_l b_{ijl}\\\sum_l b_{ijl}&\sum_l b_{jjl}\end{smallmatrix}\right) \succcurlyeq 0$.
But this constraint is redundant as it is implied by $\sum_{il} \frac{b_{iil}}{m_{iil}} \Delta_0(A_{iil}) + \sum_{i<j,l} \frac{b_{ijl}}{m_{ijl}} (\Delta_0(A_{ijl})+\Delta_0(A_{jil})) \succcurlyeq 0$.

Since $|\cC| = \sum_i b_{ii0}$ and $|\cC|^2 = \sum_{ijl} b_{ijl}$, the inequality $\sum_{ijl} b_{ijl} \ge (\sum_i b_{ii0})^2$ is valid and, using again the Schur complement, can be expressed as $\left(\begin{smallmatrix}1&\sum_{i} b_{ii0}\\\sum_{i} b_{ii0}&\sum_{ijl} b_{ijl}\end{smallmatrix}\right) \succcurlyeq 0$.
This constraint can be sharpened by considering pairs of fibers.
On the one hand, we have $x_i+x_j=b_{ii0}+b_{jj0}$.
On the other hand, we have $(x_i+x_j)^2=x_i^2+2x_ix_j+x_j^2=\sum_l b_{iil}+2\sum_l b_{ijl}+\sum_l b_{jjl}$.
The Schur complement shows then that $\left(\begin{smallmatrix}1&b_{ii0}+b_{jj0}\\b_{ii0}+b_{jj0}&\sum_l b_{iil}+2\sum_l b_{ijl}+\sum_l b_{jjl}\end{smallmatrix}\right) \succcurlyeq 0$ is equivalent to $\sum_l b_{iil}+2\sum_l b_{ijl}+\sum_l b_{jjl} \ge (b_{ii0}+b_{jj0})^2$.

Using Equations~(\ref{eq:XAm}) and~(\ref{bij}), we have 
\begin{align*}
 \frac{b_{ijl}}{m_{ijl}} \Delta_s(A_{ijl}) + \frac{b_{jil}}{m_{jil}} \Delta_s(A_{jil}) = \frac{b_{ijl}}{m_{ijl}} \left( \Delta_s(A_{ijl}) + \Delta_s(A_{jil}) \right)
\end{align*}
for $i \ne j$, which is a symmetric matrix.
Hence, using only $b_{ijl}$ for $i \le j$ fulfills the condition of SDPs to consist of symmetric matrices.

The complete SDP is given by the general conditions
\begin{align*}
\max \sum_i b_{ii0} & \text{ subject to} \\
\sum_{il} \frac{b_{iil}}{m_{iil}} \Delta_s(A_{iil}) + \sum_{i<j,l} \frac{b_{ijl}}{m_{ijl}} (\Delta_s(A_{ijl})+\Delta_s(A_{jil})) \succcurlyeq 0 &\text{ for all } s \\
\left(\begin{smallmatrix}1&b_{ii0}\\b_{ii0}&\sum_l b_{iil}\end{smallmatrix}\right) \succcurlyeq 0 &\text{ for all } i \\
\left(\begin{smallmatrix}1&b_{ii0}+b_{jj0}\\b_{ii0}+b_{jj0}&\sum_l b_{iil}+2\sum_l b_{ijl}+\sum_l b_{jjl}\end{smallmatrix}\right) \succcurlyeq 0 &\text{ for all } i<j \\
b_{ijl} \in \RR &\text{ for all } i\le j,l
\intertext{and the problem specific conditions are given by}
0 \le b_{ijl} \le A_q(n,2\lceil d/2\rceil;i) \cdot A_q(n,2\lceil d/2\rceil;j) &\text{ for all } i \le j,l \text{ with } i \ne j \text{ or } 1 \le l \\
0 \le b_{ii0} \le A_q(n,2\lceil d/2\rceil;i) &\text{ for all } i \\
b_{ijl}=0 \text{ for all } i \le j,l \text{ satisfying } i \ne j \text{ or } 1 \le l &\text{ if } l < \min\{i,j\}+(d-i-j)/2. \\
\end{align*}
This SDP is bounded and the assignment $b_{ijl} = 0$ for all $i\le j,l$ is a feasible solution.
Although $A_q(n,2\lceil d/2\rceil;k)$ is often not known explicitly, it can be replaced by a suitable upper bound, cf. \url{http://subspacecodes.uni-bayreuth.de/} associated with~\cite{HKKW2016Tables}.

The restriction of the variables in the SDP to a subset of the fibers implies the following
\begin{lemma}
Let $K$ be a subset of $\{0, \ldots, n\}$.
If $i$ and $j$ in the SDP above are restricted to values in $K$, then the optimal value of this SDP is an upper bound for $A_q(n,d;K)$.
\end{lemma}

\section{Theorem~\ref{thm:A_q74_bnd} and Related Results}\label{sec:main_thm}

Throughout this section let $\scrC$ be a subspace code of $\FF_q^7$ with minimum distance $4$.
We denote the number of elements of $\scrC$ in the $i$-th fiber (so of dimension $i$) by $x_i$.
By Theorem~\ref{thm:sdp} and Table~\ref{table:irred_reps}, we obtain a semidefinite program.
Optimizing this program with the SDP solver SDPA-GMP, we verified Theorem~\ref{thm:A_q74_bnd}.
The purpose of this section is to motivate Theorem~\ref{thm:A_q74_bnd} and provide some partial
results which might show Theorem~\ref{thm:A_q74_bnd} for all $q$.

First let us note the following result for the inner distributions of $\scrC$ in the binary case:
\begin{lemma}\label{lem:dist_subspace_code}
  Let $\scrC$ be a subspace code of $\FF_2^7$ with $384 \leq |\scrC| \leq 388$ and minimum distance $4$, then 
  one of the following occurs (up to orthogonality):
  \begin{align*}
    &|\scrC| = 388 \text{ and } x_2 = 41,\, x_4 = 347,\\
    &|\scrC| = 387 \text{ and } x_2 = 41 - \alpha,\, x_4 = 346+\alpha \text{ for } \alpha \in \{ 0, 1, \ldots, 5 \},\\
    &|\scrC| = 386 \text{ and } x_2 = 41 - \alpha,\, x_4 = 345+\alpha \text{ for } \alpha \in \{ 0, 1, \ldots, 12 \},\\
    &|\scrC| = 385 \text{ and } x_2 = 41 - \alpha,\, x_4 = 344+\alpha \text{ for } \alpha \in \{ 0, 1, \ldots, 18 \},\\
    &|\scrC| = 384 \text{ and } x_2 = 41 - \alpha,\, x_4 = 343+\alpha \text{ for } \alpha \in \{ 0, 1, \ldots, 23 \} \text{ or}\\
    &|\scrC| = 384 \text{ and } x_2 = 38 - \alpha,\, x_4 = 345+\alpha,\, x_6 = 1 \text{ for } \alpha \in \{ 0, 1, 2 \}.\\
  \end{align*}
  If $|\scrC| = 388$, then $(b_{241}, b_{442})$ is one of the following:
  \begin{align*}
    & (5026, 44058),\\
    & (5027, 44054+x) \text{ for } x \in \{ 0, \ldots, 3\},\\
    & (5028, 44051+x)  \text{ for } x \in \{ 0, \ldots, 4\},\\
    & (5029, 44047+x)  \text{ for } x \in \{ 0, \ldots, 6\},\\
    & (5030, 44044+x)  \text{ for } x \in \{ 0, \ldots, 7\},\\
    & (5031, 44042+x)  \text{ for } x \in \{ 0, \ldots, 7\},\\
    & (5032, 44039+x)  \text{ for } x \in \{ 0, \ldots, 9\},\\
    & (5033, 44037+x)  \text{ for } x \in \{ 0, \ldots, 9\},\\
    & (5034, 44035+x)  \text{ for } x \in \{ 0, \ldots, 9\},\\
    & (5035, 44033+x)  \text{ for } x \in \{ 0, \ldots, 9\},\\
    & (5036, 44032+x)  \text{ for } x \in \{ 0, \ldots, 8\},\\
    & (5037, 44031+x)  \text{ for } x \in \{ 0, \ldots, 8\},\\
    & (5038, 44030+x)  \text{ for } x \in \{ 0, \ldots, 7\},\\
    & (5039, 44029+x)  \text{ for } x \in \{ 0, \ldots, 6\},\\
    & (5040, 44029+x)  \text{ for } x \in \{ 0, \ldots, 4\},\\
    & (5041, 44029+x)  \text{ for } x \in \{ 0, 1, 2\},\\
    & (5042, 44029+x)  \text{ for } x \in \{ 0, 1\}.\\
  \end{align*}
\end{lemma}
To obtain this result, we solve the SDP as described in Subsection~\ref{subsec:sdp} and added additional constraints
which forced certain distributions for the $x_i$. For $|\scrC|=388$ we additionally
determined all possible distributions of the $b_{ijk}$'s using the same idea.
This ruled out $x_2 = 40$ and $x_4 = 248$ (which is otherwise feasible).

We use $x_2 \leq A_q(7,4;2) = q^5+q^3+1$ and $x_5 \leq A_q(7,4;5) = A_q(7,4;2) = q^5+q^3+1$. This is implied by the following lemma due to Beutelspacher and orthogonality.

\begin{lemma}[\cite{Beutelspacher1975}]\label{lem:linespread}
$A_q(n,2k;k) = \frac{q^n-q}{q^k-1}-q+1$ if $k$ divides $n-1$.
\end{lemma}

The following lemma generalizes $x_3+x_4 \le 381$ in the binary case from~\cite[Lemma~4.2.ii]{Honold2016}.

\begin{lemma}\label{lem:rel34}
  We have $x_3 + x_4 \leq (q^2-q+1) \gauss{7}{1}$ with equality only if $x_3 = 0$ or $x_4 = 0$.
\end{lemma}
\begin{proof}
    We write $b=x_3$ and $c=x_4$ to avoid indices.
    The only allowed relations are (up to transposition and orthogonality) $R_{330}, R_{332}, R_{333}, R_{342}, R_{343}$.
    Let $x_3 \beta$ denote the number of pairs in relation $R_{332}$, $\delta$ the number of pairs in relation $R_{342}$,
    $x_4 \gamma$ the number of pairs in relation $R_{442}$.
    From $\Delta_1(A_{abc})$ and, respectively, $\Delta_2(A_{abc})$ and Theorem~\ref{thm:sdp} we obtain the following positive semidefinite matrices
    (after some simplifications and multiplying by $q^3 \sqrt{q} \psi \gauss{3}{1} \gauss{4}{1} \gauss{5}{1} \gauss{7}{1}$):
    \begin{align*}
        &N_{1} = 
        \begin{pmatrix}
            b q^3 ( \gauss{3}{1} \gauss{7}{1} - \gauss{3}{1}^2 b + \beta \gauss{7}{1} ) &
            q^{5/2} (\gauss{7}{1} \delta - bc \gauss{3}{1} \gauss{4}{1})
            \\
            q^{5/2} (\gauss{7}{1} \delta - bc \gauss{3}{1} \gauss{4}{1})
            &
            c q^3 ( \gauss{3}{1} \gauss{7}{1} - \gauss{3}{1}^2 c + \gamma \gauss{7}{1} )
        \end{pmatrix}\\
        &N_2 = 
        \begin{pmatrix}
          b q \gauss{2}{1} (\gauss{3}{1} (q^7+q^5+b-1) - \beta \gauss{2}{1}^2 \psi)
          & 
          -\gauss{2}{1}\gauss{3}{1} ( (q^3+1) \delta - \varphi bc)
          \\
          -\gauss{2}{1}\gauss{3}{1} ( (q^3+1) \delta - \varphi bc)
          & 
          c q \gauss{2}{1} (\gauss{3}{1} (q^7+q^5+c-1) - \gamma \gauss{2}{1}^2 \psi)
        \end{pmatrix}
    \end{align*}
    For an $m \times m$ matrix $M$ and a set $I$, let $M_{I}$ denote the $m \times m$ with
    $(M_{I})_{xy} = M_{xy}$ if $x,y \in I$ and $(M_{I})_{xy} = 0$ otherwise.
    We set $N_t = N_1 + t_1 N_2 + t_2 ((N_2)_{\{1\}} + (N_2)_{\{2\}})$,
    where
    \begin{align*}
        &t_1 = \frac{q^{5/2} \gauss{7}{1} }{\gauss{2}{1} \gauss{6}{1}}, && t_2 = \frac{ q^2 \gauss{5}{1} \gauss{7}{1} }{ \gauss{2}{1}^2 \gauss{6}{1} (q^2+q^{3/2}+q+q^{1/2}+1)}
    \end{align*}
    For $q \geq 2$ the factors $t_1, t_2$ are positive, so $N_t$ is a positive semidefinite matrix.
    Hence, $\det(N_t) \geq 0$. Rearranging for $b$ yields
    \begin{align*}
     0 \leq b \leq ((q^2-q+1) \gauss{7}{1} - c) \frac{1}{1 + \frac{c}{q \psi \gauss{3}{1}^2}}.
    \end{align*}
    This implies the assertion.
\end{proof}

This can be improved to:

\begin{corollary}\label{lem:rel134}
We have $x_1 + x_3 + x_4 \leq (q^2-q+1) \gauss{7}{1}$ with equality only if $x_3 = 0$ or $x_4 = 0$.
\end{corollary}
\begin{proof}
The minimum distance implies $x_1 \le 1$. If $x_1=0$, then Lemma~\ref{lem:rel34} shows the claim. Hence, we assume $x_1=1$.

The only allowed relations are (up to transposition and orthogonality) $R_{110}$, $R_{131}$, $R_{141}$, $R_{333}$, $R_{332}$, $R_{330}$, $R_{343}$, and $R_{342}$.
Let $(x_3^2-x_3)a_{332}$ denote the number of pairs in relation $R_{332}$, $(x_4^2-x_4)a_{442}$ the number of pairs in relation $R_{442}$, and $x_3 x_4 a_{342}$ the number of pairs in relation $R_{342}$.
From $\Delta_1(A_{abc})$ and, respectively, $\Delta_2(A_{abc})$ and Theorem~\ref{thm:sdp} we obtain the following positive semidefinite matrices (after some simplifications and multiplying by $\gauss{7}{1}$):

\begin{align*}
N_1=\resizebox{0.95\hsize}{!}{$
\left(\begin{matrix}
1&
-{\frac {x_3}{\sqrt {[5] \varphi} ( q^5+q^2 ) }}&
-{\frac {x_4 \sqrt {\varphi}}{\sqrt {[5]} q^{5/2}[3]  \psi }}\\
-{\frac {x_3}{\sqrt {[5] \varphi} ( q^5+q^2 ) }}&
{\frac {x_3  ( [7] [3] -a_{332} [7] +x_3( a_{332} [7] -q^2 -[4] -[5] +1 ) ) }{[5] q^{3} ( q^4+q^2+1 ) \varphi [2]}}&
{\frac {x_3 x_4  ( a_{342} [7]+[2]-[4]-[5]-[6]+1 ) }{ q^{7/2}[5] [3] \varphi \psi [2]}}\\
-{\frac {x_4 \sqrt {\varphi}}{\sqrt {[5]} q^{5/2}[3]  \psi }}&
{\frac {x_3 x_4  ( a_{342} [7]+[2]-[4]-[5]-[6]+1 ) }{ q^{7/2}[5] [3] \varphi \psi [2]}}&
{\frac {x_4  ( [7] [3]-a_{442} [7]+x_4(a_{442} [7]-q^2-[4]-[5]+1) ) }{[5] q^{3} ( q^4+q^2+1 ) \varphi [2]}}
\end{matrix}\right)$}
\end{align*}

\begin{align*}
N_2=\resizebox{0.95\hsize}{!}{$
\left(\begin{matrix}
0&
0&
0\\
0&
{\frac {x_3  (a_{332}(x_3-1)(q^2-[5])+[7](q^3+q-1)+[3] x_3-[5]+1) }{[5] {q}^{5} \varphi  ( q^4+q^2+1 ) }}&
-{\frac {x_3 x_4  (a_{342}(q^3+1)-\varphi) }{[5] {q}^{6} \psi  \varphi }}\\
0&
-{\frac {x_3 x_4  (a_{342}(q^3+1)-\varphi) }{[5] {q}^{6} \psi \varphi }}&
{\frac {x_4  ( a_{442}(x_4-1)(q^2-[5])+[7](q^3+q-1)+x_4 [3]-[5]+1 ) }{[5] {q}^{5} \varphi  ( q^4+q^2+1 ) }}
\end{matrix}\right)$}
\end{align*}
We set $N_t = N_1 + t_1 N_2 + t_2 ((N_2)_{\{2\}} + (N_2)_{\{3\}})$, where
\begin{align*}
&t_1=\frac { q^{5/2}[7]}{[3] \psi {[2]}^{2}},&&
t_2=\frac { [7]q^2([3]-\sqrt{q}[2]) }{[3] \psi {[2]}^{3}}.
\end{align*}
For $q \geq 2$ the factors $t_1, t_2$ are positive, so $N_t$ is a positive semidefinite matrix.
Hence, $\det(N_t) \geq 0$ and solving this inequality for $x_3$ yields an upper bound for $x_3$, say $u(q,x_4)$.
Then, the objective function is upper bounded by $1+u(q,x_4)+x_4$, which has its maximum on $0 \le x_4 \le (q^2-q+1) \gauss{7}{1}$ at $\sqrt{q\gauss{4}{1}^2(q^4+q^2+1)^2}-q([7]+q^{2}\varphi)$
with the value
$2\sqrt{q}(q([7]+q[4])-\sqrt{q}-q^{3/2}-5/2q^{5/2}-q^{7/2}-2q^{9/2}-q^{11/2}-q^{13/2}+1)$, which is at most $(q^2-q+1) \gauss{7}{1}$.

\end{proof}

\begin{lemma}\label{lem:rel23}
    We have $x_2+x_3 \leq (q^2-q+1) \gauss{7}{1}$ with equality only if $x_2=0$.
\end{lemma}
\begin{proof}
    We write $a=x_2$ and $b=x_3$ to avoid indices.
    The only allowed relations are (up to transposition and orthogonality) $R_{220}, R_{222}, R_{232}, R_{330}, R_{332}, R_{333}$.
    Let $x_3 \beta$ denote the number of pairs in relation $R_{332}$.
    From $\Delta_1(A_{abc})$ and, respectively, $\Delta_2(A_{abc})$ and Theorem~\ref{thm:sdp} we obtain the following positive semidefinite matrices:
    \begin{align*}
        &N_{1} = 
        \begin{pmatrix}
            a \gauss{4}{1} (\gauss{7}{1} - \gauss{2}{1} a) &
            -ab q^{7/2} \gauss{2}{1} \gauss{3}{1} \sqrt{\varphi}
            \\
            -ab q^{7/2} \gauss{2}{1} \gauss{3}{1} \sqrt{\varphi}
            &
            b q^3 ( \gauss{3}{1} \gauss{7}{1} - \gauss{3}{1}^2 b + \beta \gauss{7}{1} )
        \end{pmatrix}\\
        &N_2 = 
        \begin{pmatrix}
          a q^3 \gauss{2}{1} ((\psi \gauss{3}{1} (q^2 \gauss{4}{1} - 1) + a)
          & 
          ab q^2 \gauss{2}{1} \sqrt{\gauss{3}{1}}
          \\
          ab q^2 \gauss{2}{1} \sqrt{\gauss{3}{1}}
          & 
          b q \gauss{2}{1} (\gauss{3}{1} (q^7+q^5+b-1) - \beta \gauss{2}{1}^2 \psi)
        \end{pmatrix}
    \end{align*}
    Set $N_t = N_1 + t_1 N_2$, where $t_1 = \frac{q^2 \gauss{7}{1}}{\gauss{2}{1}^2 \psi}$.
    As $t_1 \geq 0$, $N_t$ is positive semidefinite, so $\det(N_t) \geq 0$.
    Rearranging this for $b$ yields
    \begin{align*}
        b \leq ((q^2-q+1) \gauss{7}{1} - a) \frac{1}{1 + a \frac{\gauss{2}{1}^2 C }{ q \gauss{5}{1}^3 }},
    \end{align*}    
    where $C = 2 \gauss{2}{1} \sqrt{ q \gauss{3}{1} \psi } - ({{q}^{4}}+3 {{q}^{3}}+3 {{q}^{2}}+3 q+1)$.
    The assertion follows.
\end{proof}

This also shows that only proper subspaces are of interest.

\begin{corollary}\label{cor:dim0167}
If $(q^2-q+1) \gauss{7}{1} +3 \le |\cC|$, then $x_0=x_7=0$ and $x_1+x_6 \le 1$.
\end{corollary}
\begin{proof}
By the minimum distance, we have $ 0 \le x_i \le 1$ for $i \in \{0,1,6,7\}$.
If $x_0=x_7=1$ then the minimum distance shows $\cC \subseteq \{\{0\},\FF_q^7\}$.
If $x_0+x_7=1$ then by orthogonality we can assume without loss of generality that $x_0=0$ and $x_7=1$ and in particular $|\cC| = x_1+x_2+x_3+1$.
If $x_1=1$ then $x_2=0$ and $|\cC| \le A_q(7,4;3)+2 \le (q^2-q+1) \gauss{7}{1} +2$ contradicting the claim.
Hence, we have $|\cC| = x_2+x_3+1 \le (q^2-q+1) \gauss{7}{1} +1$ using the inequality from Lemma~\ref{lem:rel23}.

Assume now that $x_0 = x_7 = 0$ and $x_1 = x_6 = 1$.
Then $x_2 = x_5 = 0$ by the minimum distance and $|\cC| = x_3 + x_4 + 2 \le (q^2-q+1) \gauss{7}{1} + 2$ using the inequality from Lemma~\ref{lem:rel34} and completing the proof.
\end{proof}

We finish with the motivation for the bound in Theorem~\ref{thm:A_q74_bnd}.

\begin{lemma}\label{lem:rel24}
 We have $x_2 + x_4 \leq F(q)$.
\end{lemma}
\begin{proof}
    We write $a=x_2$ and $c=x_4$ to avoid indices.
    The only allowed relations are (up to transposition and orthogonality) $R_{220}, R_{222}, R_{241}, R_{242}, R_{440}, R_{442}, R_{443}$.
    Let $\alpha$ denote the number of pairs in relation $R_{241}$,
    and $x_4 \gamma$ the number of pairs in relation $R_{442}$.
    From $\Delta_1(A_{abc})$ and, respectively, $\Delta_2(A_{abc})$ and Theorem~\ref{thm:sdp} we obtain the following positive semidefinite matrices:
    \begin{align*}
        &N_{1} = 
        \begin{pmatrix}
            a \gauss{4}{1} (\gauss{7}{1} - \gauss{2}{1} a) &
            \gauss{2}{1} \varphi (\gauss{7}{1} \alpha - ac q^3 \gauss{2}{1} \gauss{4}{1})
            \\
            \gauss{2}{1} \varphi (\gauss{7}{1} \alpha - ac q^3 \gauss{2}{1} \gauss{4}{1})
            &
            b q^3 ( \gauss{3}{1} \gauss{7}{1} - \gauss{3}{1}^2 b + \beta \gauss{7}{1} )
        \end{pmatrix}\\
        &N_2 = 
        \begin{pmatrix}
          a q^3 \gauss{2}{1} ((\psi \gauss{3}{1} (q^2 \gauss{4}{1} - 1) + a)
          & 
          q\gauss{2}{1} \sqrt{ \gauss{3}{1} } ( ac \varphi - \alpha \psi \gauss{3}{1} )
          \\
          q\gauss{2}{1} \sqrt{ \gauss{3}{1} } ( ac \varphi - \alpha \psi \gauss{3}{1} )
          & 
          b q \gauss{2}{1} (\gauss{3}{1} (q^7+q^5+b-1) - \beta \gauss{2}{1}^2 \psi)
        \end{pmatrix}
    \end{align*}
    Set $N_t = N_1 + t_1 N_2 + t_2 (N_1)_{22}$, where
    \begin{align*}
        &t_1 = \frac{q^2 \sqrt{\varphi} \gauss{7}{1}}{ \gauss{6}{1} \sqrt{\gauss{3}{1}} }, && t_2 = \frac{\gauss{2}{1}^2 \sqrt{\varphi}}{\sqrt{\gauss{3}{1}^3}} - 1.
    \end{align*}
    As $t_1, t_2 \geq 0$, $N_t$ is positive semidefinite, so $\det(N_t) \geq 0$.
    Solving this inequality for $c$ gives an upper bound on $c$ in terms of $a$, say $c(a)$.
    Then $a + c \leq \lfloor a + c(a) \rfloor$. The function $F(q)$ is defined such that $F(q) = \max_{0 \leq a \leq q^5+q^3+1} \lfloor a + c(a) \rfloor$
    for $q$ a prime power. Here we use Lemma~\ref{lem:linespread}.
\end{proof}

Combining Lemma~\ref{lem:rel23}, Lemma~\ref{lem:rel24}, and Lemma~\ref{lem:rel34} shows Theorem~\ref{thm:A_q74_bnd_gen}.

We applied also the strategy of~\cite[Section~4.1]{Honold2016} in the binary case with functions $x_3 \le f'(x_4)$, $x_3 \le g'(x_2)$, and $x_3 \le h'(x_5)$ defined by
\begin{align*}
f'(x) &= \left\lfloor\frac{294(381-x)}{294+x}\right\rfloor, \quad\quad\quad
g'(x)  = \left\lfloor\frac{62(6\sqrt{70}+59)(381-x)}{372\sqrt{70}+3658+9x}\right\rfloor, \text{ and} \\
h'(x) &= \left\lfloor\frac{(13209651-28575x)\sqrt{35}+73499853-192913x}{192913+34671\sqrt{35}-98x}\right\rfloor,
\end{align*}
as implied by the same reasoning as in Lemmata~\ref{lem:rel34},~\ref{lem:rel23}, and~\ref{lem:rel24}.
Denote the previous upper bounds $f^{\text{HKK}}$, $g^{\text{HKK}}$, and $h^{\text{HKK}}$ from~\cite[Lemma~4.2]{Honold2016},~\cite[Lemma~4.3]{Honold2016}, and~\cite[Lemma~4.4]{Honold2016}, respectively.
The bounds $f'$ and $h'$ are stronger than $f^{\text{HKK}}$ and $h^{\text{HKK}}$, respectively, for large arguments while $g^{\text{HKK}}(x) \le g'(x)$ for all $0 \le x \le 41$.
Assuming $x_4 \le x_3$, we have $x_4 \le 151$ by $f'$, improving $x_4 \le 190$ from~\cite[Lemma~4.2.i]{Honold2016}.
Then, as shown in~\cite[Section~4.1]{Honold2016}, if $x_4 \le x_3$ we have the bound
\begin{align*}
x_2+x_3+x_4+x_5
&\le
\max_{\substack{0 \le x_2 \le 41 \\ 0 \le x_5 \le 41}}
x_2+F(\min\{g(x_2),h(x_5)\},\min\{g(x_5),h(x_2)\})+x_5
\text{ with} \\
F(u_3,u_4)
&=
\max_{0 \le x_4 \le \min\{u_3,u_4,151\}}
\min\{u_3,f(x_4)\}+x_4
\end{align*}
in which we fixed an error with the $\max$ in $F$ from~\cite[Section~4.1]{Honold2016}.
Using only the functions implied by the SDP arguments, i.e., $f=f'$, $g=g'$, and $h=h'$, an exhaustive computer calculation determines the right hand side as $432$.
By taking $f=\min\{f',f^{\text{HKK}}\}$, $g=g^{\text{HKK}}$, and $h=\min\{h',h^{\text{HKK}}\}$, the right hand side of the maximization problem is $393$ which improves the $406$ from~\cite[Section~4.1]{Honold2016} but is inferior to Theorem~\ref{thm:A_q74_bnd}.
Nevertheless, this calculation involved only integer computations and is resilient against numerical errors.
Then Corollary~\ref{cor:dim0167} shows $A_2(7,4) \le 394$.

\section{New and Updated SDP Bounds}\label{sec:comps}

Bachoc et al.~\cite{Bachoc2013} provided bounds for network codes with odd distances, but not for even distances or $q>2$.
With the general formulas for triple intersection numbers described in Section~\ref{sec:triples},
we can calculate the corresponding coherent configuration with standard techniques and let a semidefinite
programming solver (here SDPA-GMP\footnote{Some numbers require a higher precision output than what SDPA offers. See \url{https://github.com/ferihr/sdpa-gmp} for a version where the constants \texttt{P\_FORMAT\_obj} and \texttt{P\_FORMAT\_gap} in \texttt{sdpa\_io.h} adjust the output length.}) find a bound on the corresponding problem.
The following tables list bounds on $A_q(n, d)$ for small $q$ and small $n$, complementing and, for $q=2$ and odd $d$, improving the work by Bachoc et al.
New best bounds are \textbf{bold}.
If $q=2$ and $d$ is odd, the new SDP bound is better than the old or there was no previous SDP bound in literature, then the entry is in \textit{italics}.

\begin{table}[H]

\begin{center}
\begin{adjustbox}{max width=\textwidth}
\begin{tabular}{rrrrrrrrrrrrrrrrrrrrr}
$d \setminus n$  & 8                 & 9               & 10                & 11                    & 12                    & 13                       & 14 \\ \midrule 
3                & \textit{\textbf{9191}}  & \textbf{107419} & \textit{\textbf{2531873}} & \textit{\textbf{57201557}}& \textit{\textbf{2685948795}}& \textit{\textbf{119527379616}} & \textit{\textbf{11215665059647}}\\
4               & \textbf{6479}  & \textbf{53710} & \textbf{1705394} & \textbf{28600778} & \textbf{1816165540} & \textbf{59763689822}  & \textit{\textbf{7496516673358}}   \\
5                & \textit{327}            & \textit{2458}         & \textit{48255}   & \textit{\textbf{660265}}  & \textit{\textbf{26309023}}  & \textit{\textbf{688127334}}    & \textit{\textbf{54724534275}}\\
6               &     260   & \textbf{1240}  & 38455   & \textbf{330133}   & \textbf{21362773}   & \textbf{344063682}    & \textit{\textbf{43890879895}}   \\
7                &                   &                 & \textit{1219}             & \textit{8844}             & \textit{\textbf{314104}}    & \textit{\textbf{4678401}}      & \textit{\textbf{330331546}}\\
8               &           &           & 1090        & 4480        & \textbf{279476}     & \textbf{2343888}      & \textit{\textbf{292988615}}   \\
9                &                   &                 &                     &                     & \textit{\textbf{4483}}      & \textit{34058}                & \textit{\textbf{2298622}}\\
10              &           &           &             &             & 4226          & \textbf{17133}        & \textit{\textbf{2164452}}   \\
11               &                   &                 &                     &                     &                       & 259                      & \textit{\textbf{17155}}\\
12              &           &           &             &             &               &                 & \textit{16642}   \\

\end{tabular}
\end{adjustbox}
\caption{SDP bounds on $A_2(n, d)$.}
\label{table:A2nd_odd}
\end{center}
\end{table}

\begin{table}[H]
\begin{center}
\begin{adjustbox}{max width=\textwidth}
\begin{tabular}{rrrrrrrr}
$d \setminus n$ & 6 & 7 & 8        & 9           & 10            & 11              & 12 \\ \midrule
3  &  967   & \textbf{15394} & 760254 & \textbf{34143770} & \textbf{5026344026} & \textbf{675225312722} & \textbf{298950313257852} \\
4  &\textbf{788}  & \textbf{7696}  & \textbf{627384} & \textbf{17071886} & \textbf{4112061519} & \textbf{337612656529} & \textbf{244829520433920} \\
5  &        & 166      & 7222      & 123535   & \textbf{16008007}   & \textbf{818518696}    & \textbf{320387589445} \\
6  &        &          & 6727      & \textbf{61962}    & \textbf{14893814}   & \textbf{409259348}    & \textbf{298571221318} \\
7  &        &          &           & 490         & 61002         &    1076052      & \textbf{400831735} \\
8  &        &          &           &             & 59539         & \textbf{539351}       & \textbf{391178436} \\
9  &        &          &           &             &               &    1462         & 537278 \\
10 &        &          &           &             &               &                 & 532903 \\
\end{tabular}
\end{adjustbox}
\caption{SDP bounds on $A_3(n, d)$.}
\label{table:A3nd}
\end{center}
\end{table}

\begin{table}[H]
\begin{center}
\begin{adjustbox}{max width=\textwidth}
\begin{tabular}{rrrrrrrr}
$d \setminus n$ & 6 & 7& 8          & 9            & 10              & 11   \\ \midrule
3  &  4772  &\textbf{142313} &20482322 &\textbf{2341621613} &\textbf{1343547758223} &\textbf{614496020025690}\\
4  &\textbf{4231} &\textbf{71156}  &\textbf{18245203} &\textbf{1170810807} &\textbf{1194101275238} &\textbf{307248010015067}\\
5  &        & 516      & 68117      & \textbf{2132181}   &\textbf{1122729102}    &\textbf{140323867490} \\
6  &        &          & 66054      & \textbf{1067796}   &\textbf{1088550221}    &\textbf{70161933745} \\
7  &        &          &            & 2052         & 1058831         & 33669242 \\
8  &        &          &            &              & 1050630         & 16847095 \\
9  &        &          &            &              &                 & 8196 \\
\end{tabular}
\end{adjustbox}
\caption{SDP bounds on $A_4(n, d)$.}
\label{table:A4nd}
\end{center}
\end{table}

\begin{table}[H]
\begin{center}
\begin{adjustbox}{max width=\textwidth}
\begin{tabular}{rrrrrrrr}
$d \setminus n$ & 6 & 7 & 8          & 9             & 10     \\ \midrule
3  & 17179  &\textbf{821170} &277100135 &\textbf{64262978412} & \textbf{108238287449582}\\
4  &\textbf{15883}&\textbf{410585} &\textbf{256754528} &\textbf{32131489207} & \textbf{100215014898311}\\
5  &        & 1254     & 398154      &\textbf{19675409}    & \textbf{31196584033}\\
6  &        &          & 391883      &\textbf{9847885}     & \textbf{30703887393}\\
7  &        &          &             & 6254          & 9803150\\
8  &        &          &             &               & 9771883\\
\end{tabular}
\end{adjustbox}
\caption{SDP bounds on $A_5(n, d)$.}
\label{table:A5nd}
\end{center}
\end{table}

\begin{table}[H]
\begin{center}
\begin{adjustbox}{max width=\textwidth}
\begin{tabular}{rrrrrrrr}
$d \setminus n$ & 6 & 7   & 8             & 9               & 10     \\ \midrule
3  & 123239  &\textbf{11807778} &14753449680 &\textbf{9728400942608} &\textbf{85039309360944189}\\
4  &\textbf{118347}&\textbf{5903889}  &\textbf{14176726504} &\textbf{4864200471305} &\textbf{81703574152063079}\\
5  &         & 4806       & 5803270       &\textbf{566262547}     &\textbf{4784663914039}  \\
6  &         &            & 5769615       &\textbf{283240686}     &\textbf{4756893963688}  \\
7  &         &            &               & 33618           &   282744208   \\
8  &         &            &               &                 &   282508875  \\
\end{tabular}
\end{adjustbox}
\caption{SDP bounds on $A_7(n, d)$.}
\label{table:A7nd}
\end{center}
\end{table}

We added these bounds and will continuously add data on the SDP bound for larger numbers on \url{http://subspacecodes.uni-bayreuth.de/}, cf.~\cite{HKKW2016Tables}.

Compared to~\cite{honold2018johnson}, in which the authors prove upper bounds based on Johnson type arguments in the mixed dimension setting, our SDP method is stronger, e.g. for $A_q(7,3)$ ($3 \le q \le 7$) and $A_2(8,3)$, and sometimes weaker, e.g. for $A_2(10,5)$, $A_2(10,6)$, $A_2(11,7)$, $A_2(13,9)$, $A_3(9,5)$, and $A_q(8,3)$ ($3 \le q \le 7$).

\section{Future Work}\label{sec:future}

An obvious open problem is to show the bound of Theorem~\ref{thm:A_q74_bnd} for general $q$.
This might be of larger interest as it is usually very hard to optimize SDP problems with parameters
except for certain special cases. For all bounds an interesting question is if we can find constructions
which match them.

In~\cite{Schrijver2005} Schrijver successfully improved the best known bounds for constant weight codes
with semidefinite programming and a symmetry condition. It is an interesting question if Schrijver's method can also
provide better bounds on constant dimension codes. The answer to this negative. The second author was informed by Hajime Tanaka
that the centralizer algebra of a $k$-space was essentially calculated by Dunkl~\cite{Dunkl1978} in 1978 and presented 
in a more explicit way by Yuta Watanabe in his Master's thesis~\cite{Watanabe2015} in 2015. With this it is easy to calculate 
Schrijver's SDP bound for constant dimension codes. Surprisingly, Schrijver's SDP bound does not 
improve any known bounds for $A_2(n, d; k)$ for any $n \leq 20$, see the blog post on this\footnote{\url{https://ratiobound.wordpress.com/2018/10/11/schrijvers-sdp-bound-for-network-codes/}, retrieved at 20:49 MET on 29/06/2019.} by the second author.

\section*{Acknowledgments} The second author thanks Jason Williford for many interesting discussions
about semidefinite programming in coherent configurations. We would like to thank the SDPA developers
for providing this excellent software, in particular Maho Nakata for the GMP version of it.


\begin{thebibliography}{10}

\bibitem{ai2016expurgation}
J.~Ai, T.~Honold, and H.~Liu.
\newblock The expurgation-augmentation method for constructing good plane
  subspace codes.
\newblock {\em arXiv:1601.01502}, 2016.

\bibitem{Bachoc2013}
C.~Bachoc, A.~Passuello, and F.~Vallentin.
\newblock Bounds for projective codes from semidefinite programming.
\newblock {\em Adv. Math. Commun.}, 7(2):127--145, 2013.

\bibitem{Beutelspacher1975}
A.~Beutelspacher.
\newblock Partial spreads in finite projective spaces and partial designs.
\newblock {\em Math. Z.}, 145(3):211--229, 1975.

\bibitem{MR3542513}
M.~Braun, T.~Etzion, P.~R.~J. {\"O}sterg{\aa}rd, A.~Vardy, and A.~Wassermann.
\newblock Existence of {$q$}-analogs of {S}teiner systems.
\newblock {\em Forum Math. Pi}, 4:e7, 14, 2016.

\bibitem{MR3398870}
M.~Braun, M.~Kiermaier, and A.~Naki{\'c}.
\newblock On the automorphism group of a binary {$q$}-analog of the {F}ano
  plane.
\newblock {\em European J. Combin.}, 51:443--457, 2016.

\bibitem{MR3198748}
M.~Braun and J.~Reichelt.
\newblock {$q$}-analogs of packing designs.
\newblock {\em J. Combin. Des.}, 22(7):306--321, 2014.

\bibitem{Brouwer1989}
A.~E. Brouwer, A.~M. Cohen, and A.~Neumaier.
\newblock {\em Distance-regular graphs}, volume~18 of {\em Ergebnisse der
  Mathematik und ihrer Grenzgebiete (3) [Results in Mathematics and Related
  Areas (3)]}.
\newblock Springer-Verlag, Berlin, 1989.

\bibitem{Delsarte1973}
P.~Delsarte.
\newblock An algebraic approach to the association schemes of coding theory.
\newblock {\em Philips Res. Rep. Suppl.}, (10):vi+97, 1973.

\bibitem{Dunkl1978}
C.~F. Dunkl.
\newblock An addition theorem for some {$q$}-{H}ahn polynomials.
\newblock {\em Monatsh. Math.}, 85(1):5--37, 1978.

\bibitem{etzion2015structure}
T.~Etzion.
\newblock On the structure of the $q$-{F}ano plane.
\newblock {\em arXiv:1508.01839}, 2015.

\bibitem{MR3799426}
T.~Etzion.
\newblock A new approach for examining {$q$}-{S}teiner systems.
\newblock {\em Electron. J. Combin.}, 25(2):24, 2018.

\bibitem{MR2589964}
T.~Etzion and N.~Silberstein.
\newblock Error-correcting codes in projective spaces via rank-metric codes and
  {F}errers diagrams.
\newblock {\em IEEE Trans. Inform. Theory}, 55(7):2909--2919, 2009.

\bibitem{MR2810308}
T.~Etzion and A.~Vardy.
\newblock Error-correcting codes in projective space.
\newblock {\em IEEE Trans. Inform. Theory}, 57(2):1165--1173, 2011.

\bibitem{MR2801584}
T.~Etzion and A.~Vardy.
\newblock On {$q$}-analogs of {S}teiner systems and covering designs.
\newblock {\em Adv. Math. Commun.}, 5(2):161--176, 2011.

\bibitem{MR3468803}
O.~Heden and P.~A. Sissokho.
\newblock On the existence of a {$(2,3)$}-spread in {$V(7,2)$}.
\newblock {\em Ars Combin.}, 124:161--164, 2016.

\bibitem{HKKW2016Tables}
D.~Heinlein, M.~Kiermaier, S.~Kurz, and A.~Wassermann.
\newblock Tables of subspace codes.
\newblock {\em arXiv:1601.02864}, 2016.

\bibitem{Fano333}
D.~Heinlein, M.~Kiermaier, S.~Kurz, and A.~Wassermann.
\newblock A subspace code of size $333$ in the setting of a binary $q$-analog
  of the {F}ano plane.
\newblock {\em Advances in Mathematics of Communications}, 13(3):457--475,
  2019.

\bibitem{Higman1975}
D.~G. Higman.
\newblock Coherent configurations. {I}. {O}rdinary representation theory.
\newblock {\em Geom. Dedicata}, 4(1):1--32, 1975.

\bibitem{Higman1976}
D.~G. Higman.
\newblock Coherent configurations. {II}. {W}eights.
\newblock {\em Geom. Dedicata}, 5(4):413--424, 1976.

\bibitem{MR898557}
D.~G. Higman.
\newblock Coherent algebras.
\newblock {\em Linear Algebra Appl.}, 93:209--239, 1987.

\bibitem{Hobart2009}
S.~A. Hobart.
\newblock Bounds on subsets of coherent configurations.
\newblock {\em Michigan Math. J.}, 58(1):231--239, 2009.

\bibitem{Hobart2014}
S.~A. Hobart and J.~Williford.
\newblock Tightness in subset bounds for coherent configurations.
\newblock {\em J. Algebraic Combin.}, 39(3):647--658, 2014.

\bibitem{MR3444245}
T.~Honold and M.~Kiermaier.
\newblock On putative {$q$}-analogues of the {F}ano plane and related
  combinatorial structures.
\newblock In {\em Dynamical systems, number theory and applications}, pages
  141--175. World Sci. Publ., Hackensack, NJ, 2016.

\bibitem{Honold2016}
T.~Honold, M.~Kiermaier, and S.~Kurz.
\newblock Constructions and bounds for mixed-dimension subspace codes.
\newblock {\em Adv. Math. Commun.}, 10(3):649--682, 2016.

\bibitem{honold2018johnson}
T.~Honold, M.~Kiermaier, and S.~Kurz.
\newblock Johnson type bounds for mixed dimension subspace codes.
\newblock {\em arXiv:1808.03580}, 2018.

\bibitem{kiermaier2016order}
M.~Kiermaier, S.~Kurz, and A.~Wassermann.
\newblock The order of the automorphism group of a binary $q$-analog of the
  {F}ano plane is at most two.
\newblock {\em Designs, Codes and Cryptography}, 86(2):239--250, 2018.

\bibitem{MR3403762}
M.~Kiermaier and M.~O. Pav{\v{c}}evi{\'c}.
\newblock Intersection numbers for subspace designs.
\newblock {\em J. Combin. Des.}, 23(11):463--480, 2015.

\bibitem{MR2796712}
A.~Kohnert and S.~Kurz.
\newblock Construction of large constant dimension codes with a prescribed
  minimum distance.
\newblock In {\em Mathematical methods in computer science}, volume 5393 of
  {\em Lecture Notes in Comput. Sci.}, pages 31--42. Springer, Berlin, 2008.

\bibitem{MR2451015}
R.~K{\"o}tter and F.~R. Kschischang.
\newblock Coding for errors and erasures in random network coding.
\newblock {\em IEEE Trans. Inform. Theory}, 54(8):3579--3591, 2008.

\bibitem{liu2014new}
H.~Liu and T.~Honold.
\newblock Poster: A new approach to the main problem of subspace coding.
\newblock In {\em 9th International Conference on Communications and Networking
  in China (ChinaCom 2014, Maoming, China, Aug.~14--16)}, pages 676--677, 2014.
\newblock Full paper available as arXiv:1408.1181.

\bibitem{MR0465509}
F.~J. MacWilliams and N.~J.~A. Sloane.
\newblock {\em The theory of error-correcting codes. {I}}.
\newblock North-Holland Publishing Co., Amsterdam-New York-Oxford, 1977.
\newblock North-Holland Mathematical Library, Vol. 16.

\bibitem{MR1725002}
K.~Metsch.
\newblock Bose-{B}urton type theorems for finite projective, affine and polar
  spaces.
\newblock In {\em Surveys in combinatorics, 1999 ({C}anterbury)}, volume 267 of
  {\em London Math. Soc. Lecture Note Ser.}, pages 137--166. Cambridge Univ.
  Press, Cambridge, 1999.

\bibitem{MR1305448}
M.~Miyakawa, A.~Munemasa, and S.~Yoshiara.
\newblock On a class of small {$2$}-designs over {${\textrm{GF}}(q)$}.
\newblock {\em J. Combin. Des.}, 3(1):61--77, 1995.

\bibitem{Schrijver2005}
A.~Schrijver.
\newblock New code upper bounds from the {T}erwilliger algebra and semidefinite
  programming.
\newblock {\em IEEE Trans. Inform. Theory}, 51(8):2859--2866, 2005.

\bibitem{MR908977}
S.~Thomas.
\newblock Designs over finite fields.
\newblock {\em Geom. Dedicata}, 24(2):237--242, 1987.

\bibitem{MR1419675}
S.~Thomas.
\newblock Designs and partial geometries over finite fields.
\newblock {\em Geom. Dedicata}, 63(3):247--253, 1996.

\bibitem{MR1379041}
L.~Vandenberghe and S.~Boyd.
\newblock Semidefinite programming.
\newblock {\em SIAM Rev.}, 38(1):49--95, 1996.

\bibitem{Watanabe2015}
Y.~Watanabe.
\newblock An algebraic study of association schemes and its applications.
\newblock Master's thesis, Tohoku University, 2015.

\bibitem{MR2894706}
M.~Yamashita, K.~Fujisawa, M.~Fukuda, K.~Kobayashi, K.~Nakata, and M.~Nakata.
\newblock Latest developments in the {SDPA} family for solving large-scale
  {SDP}s.
\newblock In {\em Handbook on semidefinite, conic and polynomial optimization},
  volume 166 of {\em Internat. Ser. Oper. Res. Management Sci.}, pages
  687--713. Springer, New York, 2012.

\end{thebibliography}
\end{document}